\documentclass[12pt,centertags,oneside]{amsart}
\RequirePackage[utf8]{inputenc}
\usepackage[left]{lineno}
\usepackage{amstext,amsthm,amscd,typearea}
\usepackage{amsmath}
\usepackage{amssymb}
\usepackage{a4wide}
\usepackage{enumitem}
\usepackage[mathscr]{eucal}
\usepackage{mathrsfs}
\usepackage{comment}
\usepackage{pdfsync}
\usepackage{pdflscape}
\usepackage{appendix}
\usepackage{multicol}
\usepackage{tikz}
\usepackage{color}
\usepackage{array}
\usepackage{booktabs}
\newcolumntype{C}[1]{>{\centering\arraybackslash}p{#1}}
\usepackage{charter}
\usepackage{color}
\usepackage{xcolor}

\usepackage[colorlinks=true,  
            linkcolor=cyan,   
            citecolor=blue,   
            urlcolor=cyan       
            ]{hyperref}

\usepackage[a4paper,width=16.2cm,top=3cm,bottom=3cm]{geometry}
\usetikzlibrary{calc}

\newtheorem{theorem}{Theorem}[section]
\newtheorem{definition}[theorem]{Definition}
\newtheorem{proposition}[theorem]{Proposition}
\newtheorem{corollary}[theorem]{Corollary}
\newtheorem{lemma}[theorem]{Lemma}
\newtheorem{remark}[theorem]{Remark}

\newcommand{\volume}{{\rm vol}}
\newcommand{\lov}{{\rm lov}}
\newcommand{\supp}{{\rm supp}}
\newcommand{\dist}{\mathop{\mathrm{dist}}\nolimits}

\newcommand{\C}{\mathbb{C}}

\renewcommand\P{\mathbb{P}}

\title{On the support of measures of large entropy for Hénon--Sibony maps}
\author[S.~Boymurodov]{Sobir Boymurodov}
   \address{V.I.Romanovskiy Institute of Mathematics, Uzbekistan Academy of Sciences, 9, Universitet str., 100174, Tashkent, Uzbekistan}
\email{sboymurodov.research$@$gmail.com}
\begin{document}
\begin{abstract}
Let $f$ be a H\'enon--Sibony map of $\mathbb{C}^k$ of  algebraic degree $d_+\geq 2$,
whose
inverse $f^{-1}$ has algebraic degree $d_-$. 
The topological entropy of $f$ is equal to 
$\log d_+^{p} = \log d_-^{k-p}$.
We show that  every ergodic \(f\)-invariant measure \(\nu\) satisfying
$
h_\nu(f)>\log \max\{ d_+^{p-1},d_-^{k-p-1}\}
$
is supported on the 
Julia set $\mathcal{J}$ {of $f$}. 
\end{abstract}

\keywords{H\'enon--Sibony map, Julia set, Green current, support of ergodic measure}
\noindent

\maketitle

{\textbf{AMS classification:}} 
32H50, 32U40, 37A35, 37F80
\section{Introduction}

\setcounter{tocdepth}{1}
The study of 
the dynamics of polynomial automorphisms
of $\mathbb C^k$
is a central topic
in complex dynamics. For background
in dimension $k=2$,
we refer the reader to the foundational works of 
Hubbard--Oberste--Vorth \cite{HOV1},
Bedford--Lyubich--Smillie \cite{bedford1,BS91}, Friedland--Milnor \cite{fried}, and Forn{\ae}ss--Sibony \cite{forn}. In this paper, we are interested in a special class of polynomial automorphisms of $\C^k$,  namely \emph{H\'enon--Sibony maps}.

Let $f$ be a polynomial automorphism of $\mathbb{C}^k$.  It can be extended to a birational map of
$\mathbb{P}^k$, 
which we also denote by $f$.
The set $I_+$ where $f$ is not defined is called the indeterminacy set
of $f$. The inverse $f^{-1}$ of $f$
is again a polynomial automorphism, and
we denote by $I_-$ the indeterminacy set of its extension  to $\mathbb{P}^k$.   
Following \cite{sibony:panorama},
we say that 
$f$ is a 
\emph{H\'enon--Sibony map} of $\mathbb{C}^k$
if $I_{\pm}$
are non-empty and satisfy
$I_+\cap I_-=\emptyset$.
A large class of polynomial automorphisms of $\mathbb{C}^k$ satisfies this property,
see for instance \cite{fried, sibony:panorama}. 
We refer to 
\cite{deThelin, dinh-decay-henon,rigidity,sibony:panorama,Vigny26}
 for the basic properties of these maps.
When $k=2$, this class reduces to the 
so-called 
\emph{H\'enon-type maps} \cite{fried},
whose study goes back to
Hénon
\cite{HenonM}
in the real case.   
Hence, H\'enon--Sibony maps naturally generalize H\'enon-type maps to higher dimensions.

For every H\'enon--Sibony map $f$, there exists a positive integer ${ 1\leq p\leq k-1}$ such that 
$\operatorname{dim} I_{+}=k-p-1$ and $\operatorname{dim} I_{-}=p-1$.  
Let $d_+\geq 2$ and $d_-$ denote the algebraic degrees of $f$ and $f^{-1}$, respectively. 
Under this condition one can prove that $d_+^p = d_-^{k-p}$.
{Let $\omega_{\mathrm{FS}}$ denote the Fubini--Study $(1,1)$-form on $\mathbb{P}^k$, normalized so that $\int_{\P^k}\omega_{\textrm{FS}}^k=1$.}
The weak limits
\[
T_{\pm}:=\lim_{n\to\infty} d_{\pm}^{-n}(f^{\pm n})^*(\omega_{\textrm{FS}})
\]
exist and define positive closed $(1,1)$-currents of mass $1$ on $\mathbb{P}^k$, called the \emph{Green currents} associated with $f$ and $f^{-1}$, respectively.
These currents admit H\"older continuous quasi-potentials outside $I_{\pm}$ and satisfy the invariance relations
$
f^{*}(T_{+})=d_{+}T_{+}$
and
$f_{*}(T_{-})=d_{-}T_{-}$.
For integers $1 \leq \ell \leq p$ and $1 \leq \ell' \leq k-p$, 
define the \emph{Julia sets} of order $\ell$ for $f$ and of order $\ell'$ for $f^{-1}$ by

\[
J^+_\ell := \supp \,T_+^\ell
\qquad 
\text{ and }
\qquad
J^-_{\ell'} := \supp \,T_-^{\ell'}.
\]

Our main result is the following.

\begin{theorem}\label{thm:main}
  Let $f$ be a H\'enon--Sibony map on $\mathbb{C}^k$ of algebraic degree $d_{+}$. Assume that $\operatorname{dim} I_{-}=p-1$ and 
  let $1 \leq \ell \leq p$ be an integer.  If ${ Y}$ 
  is a  compact subset of  $\mathbb{P}^k$ such that ${Y}\cap J_{\ell}^{+}=\emptyset$, then $h_t(f,{ Y}) \leq(\ell-1) \log d_{+}$. 
  \end{theorem}

Define
$\mathcal{J}:=J_{p}^+\cap J_{k-p}^-,
$ which we call the \emph{Julia set} of $f$.
The  intersection
$T_+^p \wedge T_-^{k-p} =: \mu$
is well-defined and is called the \emph{equilibrium (or Green) measure} of $f$; it is the unique ergodic $f$-invariant probability measure whose entropy is maximal, equal to $p\log d_+$,
and its support is contained in $\mathcal J$.
The following result is a
consequence of Theorem \ref{thm:main}
 and
 the variational principle.

\begin{corollary}\label{cor:variational-consequence}
Let $f$, $d_+$ and $p$ be as in Theorem~\ref{thm:main}, and let $d_-$ be the algebraic degree of $f^{-1}$.
Let $\nu$ be an ergodic measure and let
$1\le \ell\le p$ and $1\le \ell'\le k-p$ be integers such that
\[
h_\nu(f)>
\log \max\{d_+^{\ell-1},\, d_-^{\ell'-1}\}.
\]
Then $\supp \,\nu\subseteq J_\ell^+\cap J_{\ell'}^-
$. In particular, if $
h_\nu(f)>
\log \max\{d_+^{p-1},\, d_-^{k-p-1}\}
$, 
then $\nu$ is supported on the Julia set $\mathcal{J}$ of $f$. 
\end{corollary}

Analogous results to Theorem~\ref{thm:main} and
Corollary~\ref{cor:variational-consequence}
for holomorphic endomorphisms of complex projective spaces
$\mathbb{P}^k$ were obtained in \cite{deT,Dinh-attracting-current}.
 The proof relies on an induction argument on the powers of the Green current with bounded quasi-potential on $\mathbb{P}^k$.
To prove Theorem~\ref{thm:main}, we can adopt a similar induction argument.
However, unlike the expanding case of holomorphic endomorphisms, H\'enon--Sibony maps exhibit both 
expanding and contracting directions.
As a consequence, one has to work with both Green currents $T_+$ and $T_-$. Moreover, the quasi-potentials of $T_+$ and $T_-$ are not bounded near the indeterminacy sets $I_+$ and $I_-$, respectively.
 To overcome this difficulty, we use the fact that
 the forward and backward orbits of compact sets disjoint from $I_+$ and $I_-$ remain uniformly away from these indeterminacy sets. This allows us to control the growth of the corresponding masses outside the Julia set.

We refer to \cite{BBR,BBR2026} 
for parallel results in the setting of polynomial-like maps on $\mathbb{C}^k$ and automorphisms of 
compact K\"ahler manifolds, where the induction 
 argument cannot be used and is replaced by
quantitative exponential convergence results towards
the Green currents and measure, see also Remark \ref{r:extra}.

\medskip

\textbf{Acknowledgments.}
The author would like to thank his advisors,
Fabrizio Bianchi and Karim Rakhimov, for their valuable guidance and helpful comments.

 \section{Preliminaries on H\'enon--Sibony maps}
 \subsection{Basic properties} 
Let $f$ be a polynomial automorphism of $\mathbb{C}^k$.
One can check that its inverse
$f^{-1}$
is also polynomial. 
We denote by $d_+\geq 2$ the algebraic degree of $f$ and by $d_-$ the algebraic degree of $f^{-1}$.
The \emph{Green functions} of $f$ and $f^{-1}$ are defined by
$$
G_{+}(z):=\lim _{n \rightarrow \infty} d_{+}^{-n} \log ^{+}\left\|f^n(z)\right\| 
\quad
\text { and } 
\quad
G_{-}(z):=\lim _{n \rightarrow \infty} d_{-}^{-n} \log ^{+}\left\|f^{-n}(z)\right\|,
$$
respectively,
where $\log ^{+}(\cdot):=\max \{\log (\cdot), 0\}$.
These functions
are H\"older continuous and plurisubharmonic  on $\mathbb{C}^k$, see \cite{sibony:panorama}. 
Define also
\[
K_{\pm}:=\big\{z\in\mathbb{C}^k \;:\; \{f^{\pm n}(z)\}_{n\geq 0}\ \text{is bounded}\big\}.
\]
Then we have $K_{\pm}=\{G_{\pm}=0\}$.
From now on, we identify  $f$ and $f^{-1}$ 
with
their natural extensions  to birational maps on $\mathbb{P}^k$ and we denote
by 
$I_\pm$ the indeterminacy  sets of $f^{\pm}$, respectively. They are analytic sets strictly contained in the hyperplane at infinity $L_{\infty} := \mathbb{P}^k \setminus \mathbb{C}^k$. 

\begin{definition}
 We say that
 $f$ is a
 \emph{H\'enon--Sibony map} 
 (or a 
 \emph{regular automorphism}  of $\mathbb{C}^k$)
if $I_+$ and $I_-$ are non-empty and they satisfy $I_+\cap I_-=\emptyset$.
\end{definition}

According to  \cite{sibony:panorama}, 
the weak limits
\[
T_{\pm}=\lim_{n\to\infty} d_{\pm}^{-n}(f^{\pm n})^*(\omega_{\mathrm{FS}})
\]
exist and define positive closed $(1,1)$-currents on $\mathbb{P}^k$, called the \emph{Green currents} of $f$ and $f^{-1}$, respectively.  Moreover, their restrictions to $\mathbb{C}^k$ satisfy
$
T_{\pm|\mathbb{C}^k}=dd^cG_{\pm}$,
where, as usual, we denote 
$d=\partial+\overline{\partial}$ and
$d^c=(\partial-\overline{\partial})/(2\pi i)$. 
For a H\'enon--Sibony map $f$, there exists an integer \(1\leq p\leq k-1\) such that
\(\dim I_+=k-p-1\), \(\dim I_-=p-1\), and
\(d_+^{p}=d_-^{k-p}\).  Below are some 
basic
properties of such maps,
see
for instance
\cite{DS16, rigidity, sibony:panorama}.
\begin{proposition}\label{prop:property-of-Henon-Sibony-maps}
Let $f, f^{-1}, I_+, I_-, 
K_+, K_-,  T_+, T_-$, and
    $p$ be as above. Then
    \begin{enumerate}
        \item  $f\left(L_{\infty} \backslash I_{+}\right)=f(I_-)=I_{-}$ and $f^{-1}\left(L_{\infty} \backslash I_{-}\right)=f^{-1}(I_+)=I_{+}$;
        \item $\overline{K_{\pm}}=K_{\pm}\cup I_{\pm}$ (where the closure is taken in $\mathbb{P}^k$);
        \item $I_{\mp}$ is attracting for $f^{\pm}$ and $\mathbb{P}^k\setminus \overline{K_{\pm}}$ is its attracting basin;
         \item 
        the quasi-potentials of $T_{\pm}$ are H\"older continuous outside $I_{\pm}$;
        \item 
        $f^*(T_+)=d_+T_+$ and $f_*(T_-)=d_-T_-$;
        \item $\supp \,T_{+}^{p+1}=I_{+}$ and $\supp \, T_{-}^{k-p+1}=I_{-}$.
    \end{enumerate}
\end{proposition}
 For each integer $1\leq \ell\leq p$ and $1\leq \ell'\leq k-p$,
define the Julia sets of order $\ell$ for $f$ 
and $\ell'$ for $f^{-1}$,
respectively, by
$$
J_\ell^+=\supp \,T_+^\ell
\qquad 
\text{ and }
\qquad
J_{\ell'}^-=\supp \,T_-^{\ell'}.
$$
 The intersection
$\mathcal{J}:=J_p^+\cap J_{k-p}^-$
will be referred to as the \emph{Julia set} of $f$. 
The 
\emph{equilibrium measure}
$\mu:=T_{+}^p \wedge T_{-}^{k-p}$,
is an ergodic invariant probability measure for $f$, satisfies $\supp \,\mu \subseteq 
\mathcal{J}$ 
and is the unique measure of maximal entropy $h_{\mu}(f)=p\log d_+= (k-p)\log d_-$, see \cite{deThelin, sibony:panorama}.
We refer to 
\cite{BD2024, 
dinh-decay-henon,
Guedj,
Vergamini-Hao, Wu22} for 
several statistical properties of such measure.
 The hyperbolicity  
of $\mu$ 
and its Lyapunov exponents are studied in \cite{deThelin2008}. 

\subsection{Preliminary Results}
We state and prove here the following simple result, whose consequence below
will be used in the sequel.

\begin{lemma}\label{lem:open-neighb} 
Let $f$ be a H\'enon--Sibony map of $\mathbb{C}^k$. If  $A$ is a compact subset of $\P^k$ such that $A\cap I_+=\emptyset$, then   there  exists $\theta>0$
such that 
$$\inf_{n\geq 0}\dist(f^n(A), I_+)\geq \theta.$$ 
\end{lemma}

\begin{proof}
Since $I_+$ is attracting for $f^{-1}$ by Proposition~\ref{prop:property-of-Henon-Sibony-maps}~(3), there exists an open neighbourhood $U$ of $I_+$ such that
$f^{-1}(U)\subset U$.
Since $A$ is compact and $A\cap I_+=\emptyset$, we may choose $U$ small enough so that
$
A\cap U=\emptyset .
$
We claim that
\begin{equation}\label{eq:fnAU}
f^n(A)\cap U=\emptyset
\quad \text{for all } n\geq 0.
\end{equation}
Indeed, if $f^n(a)\in U$ for some $a\in A$ and $n\in\mathbb{N}$, then using the inclusion
$f^{-1}(U)\subset U$ repeatedly gives
$
a\in f^{-n}(U)\subset U,
$
which contradicts 
the assumption $A\cap U=\emptyset$.
Thus
\eqref{eq:fnAU} holds.

Since $U$ is a neighbourhood of the compact set $I_+$, there exists $\theta>0$ such that
\begin{equation}\label{eq:dist}
\operatorname{dist}(\mathbb P^k\setminus U,I_+)\geq \theta.
\end{equation}
The assertion follows from \eqref{eq:fnAU} and \eqref{eq:dist} taking the infimum over $n$.
\end{proof}
\begin{corollary}\label{cor:bound-for-u+}
    Let $f$ be as above and  $u_+$ be a Green quasi-potential of $T_+$. Then,
    for every compact set $A$  satisfying $A\cap I_+=\emptyset$,
the sequence  $M_n:=\max_{z\in A}|(f^n)^*u_+(z)|$ is bounded.
\end{corollary}
\begin{proof}
As in the proof of Lemma~\ref{lem:open-neighb} we can find a small open neighbourhood $U$  of $I_+$  such that  
$f^n(A)\cap U=\emptyset$ for all  integers $n\geq 0$. Using this  together with the continuity of the Green quasi-potential \(u_+\) outside \(I_+\) (see Proposition~\ref{prop:property-of-Henon-Sibony-maps}~(4)), 
for every $n$
we obtain    
\begin{align*}
M_n&=\max_{z\in A}|(f^n)^*u_+(z)|=\max_{z\in A}|u_+(f^n(z))|
\leq \max_{z\in \P^k\setminus U}|u_+(z)|< +\infty.
\end{align*}
The assertion follows.
\end{proof}

\section{Proof of Theorem~\ref{thm:main}}
{Throughout this section, we fix a H\'enon--Sibony map
$f$ of $\mathbb{C}^k$
of algebraic degree $d_+\geq 2$, with inverse $f^{-1}$ of algebraic degree $d_-$, and an integer $1\le p\le k-1$ such that $\dim I_-=p-1$. We also fix the positive closed $(1,1)$-Green currents $T_{\pm}$ on $\mathbb{P}^k$ associated with the birational extensions of $f^{\pm}$, which we still denote by $f^{\pm}$.}

In what follows, 
the pairing $\langle\cdot,\cdot\rangle$  is used for the integral of a function with respect to a measure or more generally the value of a current at a test form. 
The mass of a positive closed $(s,s)$-current $S$ on $\mathbb{P}^k$ is  $\displaystyle \ \|S\|:=\langle S, \omega_{\mathrm{FS}}^{k-s}\rangle$. The notation $\lesssim$  stands for inequalities  up to a multiplicative constant.

\subsection{Decay estimates outside $\supp \,T^{\ell}_+$} 
 The following proposition is the key ingredient in the proof of Theorem \ref{thm:main}.

\begin{proposition}\label{prop:Induction-step}
 Let $1\leq \ell\leq p$ be an integer.
Assume 
that $V\subset \P^k$ is 
open 
and satisfies
$\supp\, T^{\ell}_+\cap \overline{{V}}=\emptyset$. There exists a constant \( C>0 \) such that, for all
integers \( 0 \leq i \leq \ell \) and \( 0 \leq j \leq k-\ell+i \), and for all integers \( 0 \leq m_1 \leq \cdots \leq m_j< m \), one has
\begin{align}\label{eq:first}
\left\|T_{+}^{\ell-i} \wedge\bigwedge_{s=1}^{j}\left(f^{m_s}\right)^*\left(\omega_{\mathrm{FS}}\right) \right\|_{{ V}} \leq C d_{+}^{m(i-1)}.
\end{align}   
\end{proposition}
When $j=0$, we mean that 
the left-hand side of \eqref{eq:first} 
is equal to $\|T_+^{\ell-i}\|_V$. 
\begin{proof} 
For $i,j$ as in the statement,
construct inductively 
an open neighbourhood
$W_{i,j}$
of $\overline{{V}}$ satisfying
\[
\overline{W}_{i,j}\cap \supp \,T_+^\ell=\emptyset
\qquad 
\text{and}
\qquad
W_{i,j} \Subset W_{i-1,j-1}\cap W_{i,j-1}.\]
We can choose the initial sets $W_{i,0}$ and $W_{0,j}$ so that 
$\overline V \Subset W_{i,0}=W_{0,j} \Subset \mathbb{P}^k \setminus \supp \,T_+^\ell
$.
 To prove (\ref{eq:first}), it suffices to show the inequality
\begin{align}\label{eq:two}
\left\|T_{+}^{\ell-i} \wedge\bigwedge_{s=1}^{j}\left(f^{m_s}\right)^*\left(\omega_{\mathrm{FS}}\right) \right\|_{W_{i, j}} \leq c_{i, j} d_{+}^{m(i-1)}
\end{align}
where $c_{i, j} \geq 0$ is a constant independent of $m$ and of $m_1, \ldots, m_j$. 
We proceed by induction on the pair $(i, j)$, with $0 \leq i \leq \ell$ and $0 \leq j \leq$ $k-\ell+i$ as in the statement. 

It is clear that 
\eqref{eq:two}
holds when $i=0$ and also when $j=0$. Indeed,
by construction,
the current $T_{+}^{\ell}$ does not give mass to the sets of the form
$W_{0, j}$.
In the case $j=0$, the integral on the left-hand side of \eqref{eq:two} is bounded above by \(1\), since for every $i$ we have
\[
\|T_+^{\ell-i}\|_{W_{i,0}} \leq \|T_+^{\ell-i}\|=\|\{T_+^{\ell-i}\}\|=\|\{\omega^{\ell-i}_{FS}\}\|=1.
\]
Therefore,
\eqref{eq:two}
holds in these cases 
taking $c_{i,0} = c_{0,j} = 1$.
\medskip

Assume now that \eqref{eq:two}
holds for $(i-1,j-1)$ and $(i,j-1)$. We shall prove it for $(i,j)$.  Fix a smooth cut-off function $\chi_{i,j}:\mathbb{P}^k\to[0,1]$ such that $\supp \, \chi_{i,j}\Subset W_{i-1,j-1}\cap W_{i,j-1}$ and $\chi_{i,j}\equiv 1$ on $W_{i,j}$.
 We only have to prove that
\begin{align}\label{eq:s1}
\left\|\chi_{i,j} T_{+}^{\ell-i} \wedge\bigwedge_{s=1}^{j}\left(f^{m_s}\right)^*\left(\omega_{\mathrm{FS}}\right) \right\| \leq c_{i, j} d_{+}^{m(i-1)}.
\end{align}
{Recall that} $T_+$ is  cohomologous to $\omega_{\textrm{FS}}$. {Hence, we have}  $T_+=\omega_{\textrm{FS}}+ dd^c u_+$, where $u_+$ is a Green quasi-potential of $T_+$ that is locally H\"older continuous on $\mathbb{P}^k\setminus I_+$; { see Proposition~\ref{prop:property-of-Henon-Sibony-maps}~(4).}
Thus, by Proposition~\ref{prop:property-of-Henon-Sibony-maps}~(5)
we have 
\[(f^{m_1})^*(\omega_{\textrm{FS}})=d_+^{m_1}T_+-dd^cu_+\circ f^{m_1}.\] 
Substituting this expression
into the left-hand side  of 
 \eqref{eq:s1} and  using the triangle inequality
 we obtain
\begin{equation}\label{eq:s4}
\begin{aligned}
    \text{LHS of } (\ref{eq:s1})
    \leq &  
    \left|\left\langle \chi_{i,j} T_{+}^{\ell-i} \wedge\bigwedge_{s=2}^{j}\left(f^{m_s}\right)^*\left(\omega_{\mathrm{FS}}\right), (dd^c u_+\circ f^{m_1})\wedge\omega_{\textrm{FS}}^{k-\ell+i-j}
    \right\rangle\right|\\
    &
+d_+^{m_1}\left\|\chi_{i,j} T_{+}^{\ell-i+1} \wedge\bigwedge_{s=2}^{j}\left(f^{m_s}\right)^*\left(\omega_{\textrm{FS}}\right) \right\|.
\end{aligned}
\end{equation}
Using that $\supp \,\chi_{i,j} \Subset W_{i-1,j-1}$ and the induction hypothesis on $(i-1,j-1)$ applied to the second term in the right-hand side of  \eqref{eq:s4}, 
we obtain
\begin{align}\label{eq:s6}
d_+^{m_1}\left\|\chi_{i,j} T_{+}^{\ell-i+1} \wedge\bigwedge_{s=2}^{j}\left(f^{m_s}\right)^*\left(\omega_{\mathrm{FS}}\right) \right\|\leq c_{i-1,j-1}d_+^{m(i-1)}.
\end{align}
We next estimate the first term in the right-hand side of \eqref{eq:s4}. To this end, we apply Stokes' theorem and obtain
\begin{equation}\label{eq:stoks}
\begin{aligned}
& \left|\left\langle \chi_{i,j} T_{+}^{\ell-i} \wedge 
\bigwedge_{s=2}^{j}\left(f^{m_s}\right)^*\left(\omega_{\mathrm{FS}}\right), (dd^c u_+\circ f^{m_1})\wedge\omega_{\textrm{FS}}^{k-\ell+i-j}
\right\rangle\right|\\
& \qquad = \left|\left\langle dd^c\chi_{i,j}\wedge T_{+}^{\ell-i} \wedge\bigwedge_{s=2}^{j}\left(f^{m_s}\right)^*\left(\omega_{\textrm{FS}}\right), (u_+\circ f^{m_1})\omega_{\textrm{FS}}^{k-\ell+i-j}
    \right\rangle\right|.
\end{aligned}\end{equation}
 The last integral 
is supported in $\supp \,\chi_{i,j}$. Since \(\supp \,\chi_{i,j}\) is disjoint from \(\supp \,(T_+^\ell)\), and since \(I_+\subset \supp \,(T_+^\ell)\) by Proposition~\ref{prop:property-of-Henon-Sibony-maps}~(6), it follows that \(\supp \,\chi_{i,j}\Subset \mathbb{P}^k\setminus I_+\).
Therefore, Corollary~\ref{cor:bound-for-u+} implies that there exists a constant $M>0$ such that
$|u_+\circ f^{m_1}|\leq M \quad \text{on } \supp \,\chi_{i,j}$, where the estimate is
uniform in $m_1$.
Using this together with
the inclusion
$\supp \,\chi_{i,j}\Subset W_{i,j-1}$  we obtain
that the right-hand side of 
\eqref{eq:stoks} is bounded by
\begin{align}\label{eq:induction-step-W_{i,j-1}}
M\,\|\chi_{i,j}\|_{\mathcal{C}^2}\left\|T_{+}^{\ell-i} \wedge\bigwedge_{s=2}^{j}\left(f^{m_s}\right)^*\left(\omega_{\mathrm{FS}}\right)\right\|_{W_{i,j-1}}\leq M\, c_{i,j-1}\,\|\chi_{i,j}\|_{\mathcal{C}^2}\,d_+^{m(i-1)},
\end{align}
where
in the last inequality
we used the induction hypothesis on $(i,j-1)$.
Therefore, combining 
\eqref{eq:s4}--\eqref{eq:induction-step-W_{i,j-1}} 
we obtain that \eqref{eq:s1} holds with
\begin{equation}\label{eq:iteration}
c_{i,j}:=c_{i-1,j-1}+M\,\|\chi_{i,j}\|_{\mathcal{C}^2}\,c_{i,j-1}.
\end{equation} 
Recall that $c_{i,0}=c_{0,j}=1$. Therefore,   
setting
$$C_0:=\max\{1, M\,\max_{i,j}\|\chi_{i,j}\|_{\mathcal{C}^2}\}$$ 
and {iterating \eqref{eq:iteration}, we obtain}
$
c_{i,j}\leq 2^jC_0^j\, 
$. In particular, for $i=\ell$ and $j=k$, we have $c_{\ell,k}\leq 2^kC_0^k$. Thus,
 (\ref{eq:first}) holds with
$C:=(2C_0)^k$.
\end{proof}

\subsection{Proof of Theorem~\ref{thm:main}} 
Recall that, for each integer $1\leq \ell\leq p$, the Julia set $J_{\ell}^+$ is defined as $J_{\ell}^+ := \supp \,T_+^{\ell}$. Let ${V}$ be an open neighbourhood of ${Y}$
 such that 
$\overline{{V}}\cap J_{\ell}^+=\emptyset$. By a classical argument due to Gromov \cite{MG}, one has
\begin{equation}\label{eq:entropy-gromov}
h_t(f,{Y})\le \lov(f,{V})
:=\limsup_{m\to\infty}\frac{1}{m}\log \volume(\Gamma_m^{V}),
\end{equation}
where $h_t(f,{Y})$ denotes the topological entropy of $f$ on ${ Y}$ \cite{Bowen},
and
\[
\Gamma_m^{V}=\{(z,f(z),\dots,f^{m-1}(z)):\ z\in {V}\}.
\]
It follows from standard arguments, see 
for instance \cite{MG},
that we also have
\[
k!\,\volume(\Gamma_m^{V})
=\sum_{0\le m_1,\dots,m_k\le m-1}
\left\| \bigwedge_{s=1}^k (f^{m_s})^*({\omega_\mathrm{FS}})\right\|_{ V}.
\]
Since this sum contains $m^k$ terms, it suffices to estimate each integral separately.
Applying Proposition \ref{prop:Induction-step} with  $i:=\ell$ and $j:=k$ we obtain 
\begin{equation}\label{eq: main-point}
\left\|\bigwedge_{s=1}^{k}\left(f^{m_s}\right)^*\left(\omega_{\mathrm{FS}}\right)\right\|_{V} \lesssim d_{+}^{(\ell-1) m}
\end{equation}
where the implicit constant is independent of {$m_1,\ldots,m_k$ and $m$}.
Combining this with \eqref{eq:entropy-gromov}, we obtain
\begin{align*}
    h_t(f,{Y})\leq \limsup_{m\to\infty}\frac{1}{m}\log \left(\frac{m^kd_+^{m(\ell-1)}}{k!}\right)=(\ell-1)\log d_+.
\end{align*}
This completes the proof of Theorem~\ref{thm:main}.

\begin{remark}\label{rem:3.2}
    Using the same arguments as in Proposition~\ref{prop:Induction-step} and Theorem~\ref{thm:main}, one can show that if a compact set $Y$ is disjoint from $J_{\ell'}^-$ (where  $1\leq \ell'\leq k-p$), then
\[
h_t(f^{-1},Y)\leq(\ell'-1)\log d_-.
\]
\end{remark}

\medskip

\begin{proof}[\textbf{Proof of Corollary~\ref{cor:variational-consequence}}]
First, we   show  that $\supp\, \nu \subseteq J_\ell^+$. 
Suppose, for the sake of contradiction, that
\(\supp \,\nu\not\subseteq J_{\ell}^+\). Then there exists an open set
\(U\Subset \mathbb P^k\setminus J_{\ell}^+\) with \(\nu(U)>0\). 
By the regularity of
\(\nu\),
we can choose a compact set \(Y\subset U\) with \(\nu(Y)>0\).
Applying
the variational principle 
(see, e.g., \cite{Bowen, Brin-Katok})
together with 
Theorem~\ref{thm:main}, we obtain
\[
h_{\nu}(f)\leq h_t(f,{Y})\leq (\ell-1)\log d_+,
\]
which contradicts the assumption 
$h_\nu(f)>(\ell-1)\log d_+
$.
Hence, we have \(\supp \,\nu\subseteq J_{\ell}^+\).

Applying the same argument to \(f^{-1}\), using
Remark~\ref{rem:3.2} and the identity
\(h_\nu(f^{-1})=h_\nu(f)\), we also obtain
$\supp \,\nu\subseteq J_{\ell'}^-$.
Therefore, we have
\[
\supp \,\nu\subseteq J_{\ell}^+\cap J_{\ell'}^-,
\]
 and the first assertion is proved. 
The second assertion immediately follows taking
$\ell=p$ and $\ell'=k-p$.
This completes the proof.
\end{proof}

\begin{remark}\label{r:extra}
The proof of Theorem~\ref{thm:main} is quite specific to the algebraic setting
considered here. It relies on the induction procedure involving wedge products of
the Green current \(T_+\) and on the control of its quasi-potential away from \(I_+\).
In more general settings, such as polynomial-like maps or automorphisms of compact
K\"ahler manifolds,
this induction method cannot be employed 
and 
is replaced
by
quantitative exponential convergence results towards
the Green currents and measure,
see \cite{BBR,BBR2026}.
In the
present setting, the main estimate can also be 
obtained
through suitable
quantitative
convergence results toward Green currents, such as \cite[Theorem~1.4]{Ahn18}.
\end{remark}

\end{document}